\documentclass[a4paper,reqno]{amsart}

\usepackage[english]{babel}
\usepackage[all]{xy}
\usepackage{amssymb}
\usepackage{verbatim}
\usepackage{fullpage}
\usepackage{enumerate}
\newcommand{\EV}[1] {E({#1})}
\newcommand{\VV}[1] {V({#1})}
\newcommand{\VG}[1] {V_G({#1})}
\newcommand{\VL}[1] {V_L({#1})}
\newcommand{\VK}[1] {V_{L_K}({#1})}
\newcommand{\PV}[1] {\Pi({#1})}
\newcommand{\PP}[1] {\Pi^+({#1})}
\newcommand{\tPhi} {\widetilde \Phi}
\newcommand{\tDelta}{\widetilde\Delta}
\newcommand{\yb}{\mathrm{lb}}
\newcommand{\YB}{\mathrm{LB}}

\DeclareMathOperator{\End}  {End}

\DeclareMathOperator{\Pic}  {Pic}
\DeclareMathOperator{\Supp} {Supp}
\DeclareMathOperator{\Proj} {Proj}

\theoremstyle{plain}
\newtheorem*{teoA}{Theorem A}
\newtheorem*{teoB}{Theorem B}
\newtheorem{lem}{Lemma}
\newtheorem{teo}[lem]{Theorem}
\newtheorem{prop}[lem]{Proposition}
\newtheorem{cor}[lem]{Corollary}

\theoremstyle{definition}
\newtheorem{dfn}[lem]{Definition}
\newtheorem{oss}[lem]{Remark}

\let\oldmarginpar\marginpar
\renewcommand\marginpar[1]{\-\oldmarginpar[\raggedleft\footnotesize
    #1]{\raggedright\footnotesize #1}}

\newcommand{\calC}{\mathcal C} \newcommand{\calD}{\mathcal D}
\newcommand{\calL}{\mathcal L} \newcommand{\calM}{\mathcal M}
\newcommand{\calN}{\mathcal N} \newcommand{\calO}{\mathcal O} 
\newcommand{\calX}{\mathcal X} 

\newcommand{\mN}{\mathbb N}  \newcommand{\mZ}{\mathbb Z}
  \newcommand{\mP}{\mathbb P}
\newcommand{\mQ}{\mathbb Q} 

\newcommand{\goU}{\mathfrak U}

\newcommand{\gog}{\mathfrak g}

\newcommand{\sfA}{\mathsf A}
\newcommand{\sfB}{\mathsf B} \newcommand{\sfC}{\mathsf C} \newcommand{\sfD}{\mathsf D}
\newcommand{\sfE}{\mathsf E} \newcommand{\sfF}{\mathsf F} \newcommand{\sfG}{\mathsf G}

\newcommand{\gra}{\alpha} \newcommand{\grb}{\beta}    \newcommand{\grg}{\gamma}
 \newcommand{\grl}{\lambda}  \newcommand{\grs}{\sigma}
\newcommand{\grf}{\varphi}\newcommand{\gro}{\omega} 

\newcommand{\grG}{\Gamma} \newcommand{\grD}{\Delta}  \newcommand{\grL}{\Lambda}
\newcommand{\grO}{\Omega} \newcommand{\grS}{\Sigma}

\newcommand{\mk}  {\Bbbk}

\newcommand{\ra}         {\rightarrow}
\newcommand{\lra}        {\longrightarrow}
\newcommand{\isocan}     {\simeq}
\newcommand{\vuoto}      {\varnothing}
\newcommand{\cech}       {\vee}
\newcommand{\E}          {\exists\,}
\renewcommand{\geq}      {\geqslant}
\renewcommand{\leq}      {\leqslant}
\newcommand{\senza}      {\smallsetminus}
\newcommand{\ristretto}  {\bigr|}
\newcommand{\ol}         {\overline}
\newcommand{\wt}         {\widetilde}

            \newcommand{\st}       {\, : \,}
\newcommand{\mand}     {\text{ and }}

\begin{document}

\title[Normality and non-normality of group
  compactifications]{Normality and non-normality of group
  compactifications in simple projective spaces}

\author[Bravi, Gandini, Maffei, Ruzzi]{Paolo Bravi, Jacopo Gandini,
  Andrea Maffei, Alessandro Ruzzi}


\curraddr{\textsc{Dipartimento di Matematica ``Guido
  Castelnuovo"\\ ``Sapienza" Universit\`a di Roma\\ Piazzale Aldo Moro
  5\\ 00185 Roma, Italy}}

\email{bravi@mat.uniroma1.it; gandini@mat.uniroma1.it;
  amaffei@mat.uniroma1.it; \newline ruzzi@mat.uniroma1.it}

\begin{abstract}
  If $G$ is a complex simply connected semisimple algebraic group and
  if $\grl$ is a dominant weight, we consider the compactification
  $X_\grl \subset \mP\big(\End(\VV{\grl})\big)$ obtained as the closure
  of the $G\times G$-orbit of the identity and we give necessary and
  sufficient conditions on the support of $\grl$ so that $X_\grl$ is
  normal; as well, we give necessary and sufficient conditions on the
  support of $\grl$ so that $X_\grl$ is smooth.
\end{abstract}

\maketitle

\section*{Introduction}

Consider a semisimple simply connected algebraic group $G$ over an
algebraically closed field $\mk$ of characteristic zero.  If $\grl$ is
a dominant weight (with respect to a fixed maximal torus $T$ and a
fixed Borel subgroup $B\supset T$) and if $\VV{\grl}$ is the simple
$G$-module of highest weight $\grl$, then $\End\big(\VV{\grl}\big)$ is a simple
$G\times G$-module. Let $I_\grl \in \End\big(\VV{\grl}\big)$ be the identity map
and consider the variety $X_\grl\subset \mP\big(\End(\VV{\grl})\big)$ given by the
closure of the $G\times G$-orbit of $[I_\grl]$. In \cite{Ka}, S.~Kannan
studied for which $\grl$ this variety is projectively normal, and
this happens precisely when $\grl$ is minuscule. In \cite{Ti},
D.~Timashev studied the more general situation of a sum of
irreducible representations, giving necessary and sufficient
conditions for the normality and smoothness of these
compactifications; however the conditions for normality are not completely
explicit.
In this paper we give an explicit characterization of the normality of $X_\grl$,
which allows to simplify the conditions for the smoothness as well.

To explain our results we need some notation. Let $\grD$ be the set of
simple roots (w.r.t.\ $T\subset B$) and identify $\grD$ with the
vertices of the Dynkin diagram. Define the support of $\grl$ as the
set $\Supp(\grl)=\{\gra\in \grD \st \langle \grl, \gra^\vee \rangle \neq 0
\}$.

\begin{teoA} [see Theorem~\ref{teo:normalita}]
The variety $X_\grl$ is normal if and only if $\grl$ satisfies the
following property:
\begin{itemize}
	\item[$(\star)$]  For every non-simply laced connected component $\Delta'$ of $\Delta$,
	if $\Supp(\grl)\cap \grD'$ contains a long root,
  then it contains also the short root which is adjacent to
  a long simple root.
\end{itemize}
\end{teoA}

In particular, if the Dynkin diagram of $G$ is simply laced then
$X_\grl$ is normal, for all $\grl$. In the paper we will prove the
theorem in a more general form, for simple (i.e. with a unique closed orbit)
linear projective compactifications of an
adjoint group (see section~\ref{ssez:XSigma}). We will make use of
the wonderful compactification of $G_\mathrm{ad}$, the
adjoint group of $G$, and of the results on projective normality of
these compactifications proved by S.~Kannan in \cite{Ka}. These results
hold in the more general case of a symmetric variety; however
our method does not apply to this more general situation (see
section~\ref{ssez:simmetrichenormalita}).

\begin{teoB} [see Theorem~\ref{smooth Xsigma}]
The variety $X_{\lambda}$ is smooth if and only if $\lambda$ satisfies property $(\star)$ of Theorem~A together with the following properties:
\begin{itemize}
\item [{\em i)}] For every connected component $\Delta'$ of $\Delta$, $\Supp(\lambda)\cap \Delta'$ is connected and, in case it contains a unique element, then this element is an extreme of $\grD'$;
\item [{\em ii)}] $\Supp(\lambda)$ contains every simple root which is adjacent to three other simple roots and at least two of the latter;
\item [{\em iii)}] Every connected component of $\grD\senza \Supp(\lambda)$ is of type $\sfA$.
\end{itemize}
\end{teoB}

Theorem B can be generalized to any simple and normal adjoint
symmetric variety. Following a criterion of $\mQ$-factoriality for
spherical varieties given by M.~Brion in \cite{Br2}, properties i) and
ii) characterize the $\mathbb{Q}$-factoriality of the normalization of
$X_\grl$ (see Proposition \ref{Q-fattorialita}), while property iii)
arises from a criterion of smoothness given by D.~Timashev in
\cite{Ti} in the case of a linear projective compactification of a
reductive group.

As a corollary of Theorem B, we get that $X_{\lambda}$ is smooth if
and only if its normalization is smooth.


The paper is organized as follows. In the first section we introduce
the wonderful compactification of $G_\mathrm{ad}$ and the normalization of
the variety $X_\grl$. In the second section we prove Theorem~A, and in
the third section Theorem~B.  In the last section we discuss some
possible generalizations of our results.

\section{Preliminaries}\label{sez:preliminari}

\subsection{Notation}\label{ssez:notazioni}
Recall that $G$ is semisimple and
simply connected. Fix a Borel subgroup $B\subset G$, a maximal torus
$T\subset B$ and let $U$ denote the unipotent radical of $B$. Lie
algebras of groups denoted by upper-case latin letters
($G,U,L,\ldots$) will be denoted by the corresponding lower-case german
letter ($\gog, \mathfrak u, \mathfrak l,\ldots$). Let $\Phi$ denote
the set of roots of $G$ relatively to $T$ and $\grD\subset \Phi$ the
basis associated to the choice of $B$. For all $\gra\in \grD$ let
$e_\gra, \gra^{\cech},f_\gra$ be an $\mathfrak{sl}(2)$-triple of
$T$-weights $\gra,0,-\gra$.  Let $\grL$ denote the weight lattice of
$T$ and $\grL^+$ the subset of dominant weights. For all $\gra\in\grD$,
denote by $\omega_\gra$ the corresponding fundamental weight.

If $\grl \in \grL$, recall the definition of its \emph{support}:
\[\Supp(\grl) = \{\gra \in \grD \st \langle \grl, \gra^\vee \rangle \neq 0 \}.\]

If $I \subset \grD$, define its \emph{border} $\partial{I}$, its
\emph{interior} $I^\circ$ and its \emph{closure} $\ol{I}$ as follows:
\[ \partial{I} = 
\{ \gra \in \grD \senza I \st \E \grb \in I \mathrm{ \; such \,  that \; } 
\langle \grb,\gra^\vee \rangle \neq 0\};\]
\[ I^\circ = I \senza \partial{(\grD \senza I)};\]
\[ \overline{I} = I \cup \partial{I}.\]


For $\grl \in \Lambda$, denote by $\calL_\grl$ the line bundle on
$G/B$ whose $T$-weight in the point fixed by $B$ is $-\grl$.
For $\grl$ dominant, $\VV{\grl} = \grG(G/B,\calL_\grl)^*$ is an
irreducible $G$-module of highest weight $\grl$; when we deal with
different groups we will use the notation $\VG{\grl}$.

Denote by $\PV{\grl}$ the set of weights occurring in $\VV{\grl}$ and
set $\PP{\grl}=\PV{\grl}\cap \grL^+$. Let $\grl \mapsto \grl^*$
be the linear involution of $\grL$ defined by $(\VV{\grl})^*\isocan \VV{\grl^*}$,
for any dominant weight $\grl$.

The weight lattice $\grL$ is endowed with the dominance order $\leq$
defined as follows: $\mu \leq \grl$ if and only if $\grl - \mu \in
\mN \grD$. If $\grb = \sum_{\gra \in \grD} n_\gra \gra \in \mZ \grD$,
define its \textit{support over} $\grD$ (not to be confused with the
previous one) as follows:
\[\Supp_\grD(\grb) = \{ \gra \in \grD \st n_\gra \neq 0 \}.\]

We introduce also some notations about the multiplication of
sections. Notice that, for all $\grl,\mu \in\Lambda$, $\calL_\grl
\otimes \calL_\mu =\calL_{\grl+\mu}$.  Therefore, if $\grl,\mu$ are
dominant weights and $n\in \mN$, the multiplication of sections
defines maps as follows:
$$
m_{\grl,\mu}:\VV{\grl}\times \VV{\mu} \ra \VV{\grl+\mu} \; \text{ and }
m_\grl^n : \VV{\grl} \ra \VV{n\grl}.
$$
We will also write $uv$ for $m_{\grl,\mu}(u,v)$ and $u^n$ for
$m^n_\grl(u)$.  Since $G/B$ is irreducible, $m_{\grl,\mu}$ and
$m^n_\grl$ induce the following maps at the level of projective spaces:
$$
\psi_{\grl,\mu}: \mP(\VV{\grl}) \times \mP(\VV{\mu}) \ra \mP(\VV{\grl + \mu})
\;\mand\; \psi^n_\grl: \mP(\VV{\grl}) \ra \mP(\VV{n\grl}).
$$

The following lemma is certainly well known; however we do not know any reference.

\begin{lem} \label{lem:immersioni}
Let $\grl,\mu$ be dominant weights.
\begin{enumerate}[\indent i)]
\item If $\Supp(\grl) \cap \Supp(\mu) = \vuoto$, then the map
  $\psi_{\grl,\mu}\colon \mP(\VV{\grl}) \times \mP(\VV{\mu}) \to \mP(\VV{\grl +
  \mu})$ is a closed embedding.
\item For any $n>0$, the map
$\psi_\grl^n: \mP(\VV{\grl})  \ra  \mP(\VV{n\grl})$ is a closed embedding.
\end{enumerate}
\end{lem}

\begin{proof} $i)$.
  Fix highest weight vectors $v_\grl \in \VV{\grl}$, $v_\mu \in \VV{\mu}$ and
  $v_{\grl+\mu}= v_\grl v_\mu \in \VV{\grl+\mu}$.

  If $V$ is irreducible, then $\mP(V)$ has a unique closed orbit,
  namely the orbit of the highest weight vector.  Consequently, since
  $\mP(V(\grl)) \times \mP(V(\mu))$ has a unique closed orbit, in
  order to prove the claim it suffices to prove that $\psi_{\grl,\mu}$
  is smooth in $x=([v_\grl],[v_\mu])$ and that the inverse image of
  $[v_{\grl+\mu}]$ is $x$. The second claim is clear for weight
  reasons.

  In order to prove that $\psi_{\grl,\mu}$ is smooth in $x$, consider $T$-stable
  complements $U \subset \VV{\grl}$, $V \subset \VV{\mu}$ and $W\subset
  \VV{\grl+\mu}$ of $\mk\,v_\grl$, $\mk\,v_\mu$ and $\mk\,v_{\grl+\mu}$.  So in a
  neighbourhood of $x$ the map $\psi_{\grl,\mu}$ can be described as
  $$
   \psi:U\times V \lra W \;\text{ where } \psi(u,v)= u v_\mu + v_\grl v + u v.
  $$
  The differential of $\psi_{\grl,\mu}$ in $x$ is then given by the
  differential of $\psi$ in $(0,0)$, thus it is described as follows:
  $$
  d\psi_x(u,v)=uv_\mu+v_\grl v.
  $$
 Suppose that $d\psi_x$ is not injective. Since it is $T$-equivariant, consider a
  maximal weight $\eta \in \PV{\grl+\mu}\senza \{\grl+\mu\}$ such that
  there exists a couple of non-zero $T$-eigenvectors $(u,v)\in \ker d
  \psi_x$ with weights respectively $\eta - \mu$ and $\eta - \grl$.
  Suppose that $\eta - \mu \in \PV{\grl}\senza \{\grl\}$ is not
  maximal and take $\gra \in \grD$ such that $\eta - \mu + \gra \in
  \PV{\grl}\senza \{\grl\}$ and $e_\gra u\neq 0$: then
\[
  (e_\gra u)v_\mu + v_\grl (e_\gra v) = e_\gra(u v_\mu + v_\grl v) = 0
\]
and $\eta + \gra \in \PV{\grl + \mu}\senza\{\grl+\mu\}$, against the
maximality of $\eta$. Thus $\eta - \mu$ is maximal in $\PV{\grl}
\senza \{\grl\}$ and similarly $\eta - \grl$ is maximal in $\PV{\mu}
\senza \{\mu\}$.  Therefore, on one hand it must be
\[ \eta - \mu = \grl - \gra \]
with $\gra \in \Supp(\grl)$, while on the other hand it must be
\[ \eta - \grl = \mu - \grb \]
with $\grb \in \Supp(\mu)$. Since $\Supp(\grl) \cap \Supp(\mu) =
\vuoto$, this is impossible and shows that, if $(u, v) \in \ker d
\psi_x$, then it must be $u = 0$ or $v = 0$. Suppose now that $(u,0)
\in \ker d\psi_x$: then $u v_\mu = 0$ and by the irreducibility of
$G/B$ also $u=0$. A similar argument applies if $v=0$.

$ii).$ Suppose that $v,w \in V(\grl)$ are such that $v^{n} = w^{n}$:
then $v=tw$ for some $t\in \mk$. Thus $\psi_\grl^n$ is injective.  Let
us show now that $\psi_\grl^n$ is smooth; it is enough to show it in
$x = [v_\grl]$ where $v_\grl\in V(\grl)$ is a highest weight
vector. Let $V \subset V(\grl)$ be the $T$-stable complement of $\mk
v_\grl$, identified with the tangent space $T_{x} \mP(V(\grl))$. If $v
\in V$, the differential $d (\psi_\grl^n)_{x}$ is described as follows
\[
d (\psi^n_\grl)_{x} (v) = n v_\grl^{n-1} v.
\]
Thus $d (\psi^n_\grl)_{x}$ is injective and $\psi_\grl^n$ is smooth.
\end{proof}

\subsection{The variety $X_\grl$}\label{ssez:Xlambda}
If $\grl$ is a dominant weight, denote by $\EV{\grl}$ the $G\times
G$-module $\End(\VV{\grl})$ and set $X_\grl$ the closure of the
$G\times G$-orbit of $[I_\grl] \in \mP(\EV{\grl})$. More
generally if $\grl_1,\dots,\grl_m$ are dominant weights we define
\[X_{\grl_1, \ldots, \grl_m}
= \ol{G\times G([I_{\grl_1}], \ldots, [I_{\grl_m}])} \subset
\mP(\EV{\grl_1}) \times \cdots \times \mP(\EV{\grl_m}).
\]
Since $\EV{\grl}$ is an irreducible $G\times G$-module of highest
weight $(\grl,\grl^*)$, as a consequence of Lemma \ref{lem:immersioni}
we get that if $\grl$ and $\mu$ have non-intersecting supports and if
$n\in \mN$ then
$$
X_{\grl+\mu}\simeq X_{\grl,\mu} \qquad \mathrm{ and } \qquad X_{n\grl}\isocan X_\grl.
$$
As a consequence we get the following proposition:

\begin{prop}\label{prp:supporto}Let $\grl,\mu$ be dominant weights. Then
$X_\grl \simeq  X_\mu$ as $G\times G$-varieties
if and only if $\grl$ and $\mu$ have the same support. Moreover, if
$\Supp(\grl)=\{\gra_1,\dots,\gra_m\}$ then
$$
X_\grl \isocan X_{\omega_{\gra_1},\dots,\omega_{\gra_m}}.
$$
\end{prop}

\begin{proof} By the discussion above we have to prove only that the
condition is necessary. This follows by noticing that if $X_\grl$ and
$X_\mu$ are $G\times G$-isomorphic then also their closed $G\times G$-orbits
are isomorphic, which is equivalent to the fact that
$\grl$ and $\mu$ have the same support.
\end{proof}

\subsection{The wonderful compactification of $G_{\mathrm{ad}}$ and the normalization of $X_\grl$}
\label{ssez:meravigliosa}

When $\grl$ is a regular weight (i.e.\ $\Supp (\grl)=\grD$) the
variety $X_\grl$ is called the wonderful compactification of $G_\mathrm{ad}$
and it has been studied by C.~De Concini and C.~Procesi in
\cite{CP}. We will denote this variety by $M$: it is smooth and the
complement of its open orbit is the union of smooth prime divisors with
normal crossings whose intersection is the closed orbit.
The closed orbit of $M$ is isomorphic to $G/B \times G/B$
and the restriction of line
bundles determines an embedding of $\Pic(M)$ into $\Pic(G/B\times G/B)$,
that we identify with $\grL \times \grL$ as before; the image of this
map is the set of weights of the form $(\grl,\grl^*)$. Therefore
$\Pic(M)$ is identified with $\grL$ and we denote by $\calM_{\grl}$ a
line bundle on $M$ whose restriction to $G/B \times G/B$ is isomorphic
to $\calL_\grl \boxtimes \calL_{\grl^*}$.  If $D\subset M$ is a
$G\times G$-stable prime divisor then the line bundle defined by $D$ is of the form
$\calM_{\gra_D}$, where $\gra_D$ is a simple root.
The map $D\mapsto \gra_D$ defines a bijection between the
set of $G\times G$-stable prime divisors and $\grD$, and we denote by $M_\gra$
the prime divisor which corresponds to a simple root $\gra$. We denote by
$s_\gra$ a section of $\calM_\gra$ whose associated divisor is $M_\gra$;
notice that such a section is $G\times G$-invariant. More
generally if $\nu = \sum_{\gra \in \grD} n_{\gra}\gra \in \mN \grD$,
set $s^\nu = \prod_{\gra \in \grD} s_{\gra} ^{n_{\gra}} \in
\grG(M,\calM_\nu)$. Then, given any $\grl\in \grL$, the multiplication by
$s^\nu$ injects $\grG(M,\calM_{\grl - \nu})$ in $\grG(M,\calM_\grl)$.

If $\grl$ is a dominant weight, the map $G_\mathrm{ad}\lra\mP(\EV{\grl})$
extends to a map $q_\grl:M\lra \mP(\EV{\grl})$ (see \cite{CP}) whose
image is $X_\grl$ and such that
$\calM_\grl=q_\grl^*(\calO_{\mP(\EV{\grl})}(1))$. If we pull back
the homogeneous coordinates of $\mP(\EV{\grl})$ to
$M$, we get then a submodule of $\grG(M,\calM_\grl)$
which is isomorphic to $\EV{\grl}^*$; by abuse of notation we will denote this
submodule by $\EV{\grl}^*$.

If $\grl \in \grL$, in \cite[Theorem 8.3]{CP} the following
decomposition of $\grG(M,\calM_\grl)$ is given:
\[\grG(M,\calM_\grl) = \bigoplus_{\mu\in \grL^+ \st \mu \leq \grl} s^{\grl-\mu}\EV{\mu}^*.\]

Consider the graded algebra $A(\grl)=\bigoplus_{n=0}^\infty A_n(\grl)$,
where $A_n(\grl)= \grG(M,\calM_{n\grl})$, and set $\widetilde{X}_\grl=
\Proj A(\grl)$. We have then a commutative diagram as follows:
\[\xymatrix{M \ar@{->>}^{p_\grl}[r] \ar@{->>}_{q_\grl}[dr]
& \widetilde{X}_\grl \ar@{->>}^{r_\grl}[d]\\  & X_\grl}\]

In \cite{Ka}, it has been shown that $A(\grl)$ is generated in degree
$1$ and in \cite{DC} that $r=r_\grl$ is the normalization of $X_\grl$.
Notice that the projective coordinate ring of $X_\grl \subset
\mP(\EV{\grl})$ is given by the graded subalgebra $B(\grl)=\bigoplus_{n=0}^\infty
 B_n(\grl)$ of $A(\grl)$ generated by $\EV{\grl}^* \subset \grG(M,\calM_\grl)$.

\subsection{The variety $X_\Sigma$} \label{ssez:XSigma}
We consider now a generalization of the variety $X_\grl$. Let $\Sigma$
be a finite set of dominant weights and denote $\EV{\Sigma} =
\bigoplus_{\mu \in \Sigma} \EV{\mu}$; let
$x_\Sigma=[(I_\mu)_{\mu\in\Sigma}]\in \mP(\EV{\Sigma})$ and define
$X_\Sigma$ as the closure of the $G\times G$-orbit of $x_\Sigma$ in
$\mP(\EV{\Sigma})$. If $\Sigma = \{\grl\}$, then we get the variety
$X_\grl$, while if $\grS = \PP{\grl}$ we get its normalization
$\widetilde X _\grl$.  Notice that the diagonal action of $G$ fixes
the point $x_\Sigma$ so we have a $G\times G$ equivariant map $G\lra
X_\Sigma$ given by $g \longmapsto (g,1)x_\Sigma$.  This map induces a
map from $G_{\mathrm{ad}}$ to $X_\Sigma$ if and only if the action of the
center of $G\times G$ on $E(\grl)$ is the same for all $\grl \in
\Sigma$ or equivalently if $\grS$ is contained in a coset of $\grL$
modulo $\mZ \grD$. In this case we say that $X_\Sigma$ is a
\emph{semi-compactification} of $G_{ad}$.  If $G_{\mathrm{ad}}$ is a simple
group and and $\Sigma \neq \{ 0 \}$ then $X_\Sigma$ is a compactification
of $G_{\mathrm{ad}}$, while if $G_{\mathrm{ad}}$ is not simple we can only say that is a
compactification of a group which is a quotient of $G_{\mathrm{ad}}$.

We say that $\Sigma$ is \emph{simple} if there exists $\grl\in\Sigma$
such that $\Sigma \subset \PP{\grl}$ or equivalently if $\grS$
contains a unique maximal element with respect to the dominance order
$\leq$.  Notice also that if $\grl \in \Sigma$ is such that for all
$\mu \in \Sigma$ different from $\grl$ the vector $\mu - \grl$ is not
in $\mQ_{\geq 0}[\grD]$ then is easy to construct a cocharacter $\chi
: \mk^*\lra G\times G$ such that $\lim_{t\to 0}\chi(t)x_\Sigma $ is
the highest weight line in $\mP(E(\grl))$. In particular $X_\Sigma$ is
a simple $G\times G$ semi-compactification of $G_{\mathrm{ad}}$ if and only if
$\Sigma$ is simple.

By the description of the normalization of $X_\grl$ is $\Sigma$ is
simple and $\grl \in \grS$ is the maximal element, then we get
$$\xymatrix{\widetilde X_\grl\ar[r]^{r} & X_\Sigma \ar[r] & X_\grl}$$
In particular, it follows that $r=r_\Sigma:\widetilde X_\grl \lra X_\Sigma$ is
the normalization of $X_\Sigma$. 

If $\grS$ is simple, denote $B(\Sigma)=\bigoplus_{n=0}^\infty
B_n(\Sigma)$ the projective coordinate ring of $X_\Sigma\subset
\mP(\EV{\Sigma})$: it is the subalgebra of $A(\grl)$ generated by
$\EV{\Sigma}^*\subset \grG(M,\calM_\grl)$.

\begin{oss}
The discussion above and the fact that in $\mP(E(\grl))$ there is only
one point fixed by the diagonal action of $G$ (the line of scalar
matrices) proves that any $G\times G$ linear projective
compactification of $G_{\mathrm{ad}}$ is of the form $X_\Sigma$. A
projective $G\times G$-variety $X$ is said to be \emph{linear} if
there exists an equivariant embedding $X\subset \mP(V)$ where $V$ is a
finite dimensional rational $G\times G$-module.  In particular as a
consequence of Sumihiro's Theorem (see for example \cite[Corollary
  2.6]{KKLV}) all normal projective compactifications are linear. In
this paper we study only linear compactifications.
\end{oss}

\section{Normality}\label{sez:normalita}

In this section we determine for which simple $\Sigma$ the variety
$X_\Sigma$ is normal, proving in particular Theorem~A. In the
following, by $\lambda$ we will always denote the maximal element of
$\Sigma$.

Let $\grf_\grl \in \EV{\grl}^*$ be a highest weight vector and set
$X_\Sigma^\circ \subset X_\Sigma$ the open affine subset defined by the
non-vanishing of $\grf_\grl$. In particular, we set $\wt X_\grl =
X_{\PP{\grl}}$ and notice that $\wt X ^\circ_\grl =
r^{-1}(X_\Sigma^\circ)$. Notice that $X^\circ_\Sigma$ is $B\times B$-stable and,
since it intersects the closed orbit, it intersects every orbit:
therefore $X_\Sigma$ is normal if and only if $X_\Sigma^\circ$ is
normal if and only if the restriction $r\ristretto_{\wt X
^\circ_\grl} : \wt X^\circ_\grl \to X^\circ_\Sigma$ is an isomorphism.
Denote by $\bar B(\Sigma)$ the coordinate ring of
$X^\circ_\Sigma$ and by $\bar A(\grl)$ the coordinate ring of
$\wt X^\circ_\grl$; then we have
$$
\bar A(\grl)= \{\frac{\grf}{\grf_\grl^n}\st \grf\in A_n(\grl)\}
\supset \{\frac{\grf}{\grf_\grl^n}\st \grf\in B_n(\Sigma)\} = \bar B(\Sigma)
$$
and $X_\Sigma$ is normal if and only if $\bar A(\grl)=\bar B
(\Sigma)$. The rings $\bar A(\grl)$ and $\bar B(\Sigma)$ are not $G\times G$-modules, however
since $X^\circ_\Sigma$ is an open subset of $X_\Sigma$ we still have an
action of the Lie algebra $\gog\oplus\gog$ on them.

By \cite{Ka}, $\bar A(\grl)$ is generated by the elements of the
form $\grf/\grf_\grl$ with $\grf\in A_1(\grl)$. In particular we have the
following lemma.

\begin{lem}\label{lem:general-normality}
The variety $X_\Sigma$ is normal if and only if for all $\mu\in\grL^+$
such that $\mu\leq\grl$ there exists $n>0$ such that
\[s^{\grl-\mu}\EV{\mu+(n-1)\grl}^*
\subset B_n(\Sigma).\]
\end{lem}

\begin{proof} Let $\grf_\mu \in s^{\grl-\mu}\EV{\mu}^*$ be a highest weight vector
and suppose that $X_\Sigma$ is normal. Then, by the descriptions of
$\bar A(\grl)$ and $\bar B(\Sigma)$, for every dominant weight $\mu\leq \grl$
there exist $n>0$ and $\varphi\in B_n(\Sigma)$ such that
${\grf}/{\grf_\grl^n} = {\grf_\mu}/{\grf_\grl}$ or equivalently
$\grf=\grf_\mu \grf_\grl^{n-1}\in B_n(\Sigma)$. Since $\grf$ is a
highest weight vector of the module
$s^{\grl-\mu}\EV{\mu+(n-1)\grl}^*$ the claim follows.

Conversely assume that for every dominant weight $\mu\leq \grl$ there exists
$n$ such that
$$s^{\grl-\mu}\EV{\mu+(n-1)\grl}^*\subset B_n(\Sigma);$$
in particular $\grf=\grf_\mu \grf_\grl^{n-1}\in
B_n(\Sigma)$. Let's prove that $\grf/\grf_\grl \in \bar B(\grS)$ for every
$\grf \in s^{\grl-\mu}\EV{\mu}^*$; this implies the thesis since
$\bar A(\grl)$ is generated in degree one. If $\grf=\grf_\mu$ this is
clear. Using the action of the Lie algebra $\gog\oplus\gog$ on $\bar
B(\Sigma)$, let's show that if $\grf/\grf_\grl \in \bar B(\Sigma)$
then $f_\gra(\grf)/\grf_\grl \in \bar B(\Sigma)$: indeed we have
$$ \frac{f_\gra(\grf)}{\grf_\grl} = f_\gra(\frac{\grf}{\grf_\grl}) +
\frac{\grf}{\grf_\grl} \cdot \frac{f_\gra(\grf_\grl)}{\grf_\grl}
$$ and the claim follows since
$f_\gra(\grf_\grl)\in \EV{\grl}^*\subset B_1(\Sigma)$.
\end{proof}

We can describe the set $B_n(\Sigma)$ more explicitly. Indeed, as in
\cite{DC} or in \cite{Ka}, it is possible to identify sections of a line bundle on $M$
with functions on $G$ and use the description of the multiplication of
matrix coefficients. Recall that as a $G\times G$-module we have
$\mk[G]=\bigoplus_{\grl\in \grL^+}\EV{\grl}^* \isocan \bigoplus_{\grl\in
  \grL^+}\VV{\grl}^*\otimes \VV{\grl}$.  More explicitly if $V$ is a
representation of $G$, define $c_V:V^*\otimes V \lra \mk[G]$ as
usual by $c_V(\psi \otimes v)(g)= \langle \psi, gv\rangle$. If we
multiply functions in $\mk[G]$ of this type then we get
$$
c_{V}( \psi \otimes v) \cdot c_{W}(\chi \otimes w) =
c_{V\otimes W} \big((\psi\otimes\chi)\otimes(v\otimes w)\big):
$$
in particular we get that the image of the multiplication $\EV{\grl}^*
\otimes \EV{\mu}^* \lra \mk[G]$ is the sum of all $\EV{\nu}^*$ with
$\VV{\nu}\subset \VV{\grl}\otimes \VV{\mu}$.

As a consequence we obtain the following Lemma:

\begin{lem}[{\cite[Lemma 3.1]{Ka} or \cite{DC}}]\label{lem:coefficientimatriciali}
Let $\nu, \nu'$ be dominant weights, then the image of $\EV{\nu}^*
\otimes \EV{\nu'}^*$ in $\grG(M, \calM_{\nu+\nu'})$ via the
multiplication map is
\[\bigoplus_{\VV{\mu}\subset \VV{\nu}\otimes \VV{\nu'}}\!\!\!\!s^{\nu+\nu'-\mu}\EV{\mu}^*.\]
\end{lem}

\begin{proof}
Indeed let $ \pi : G \ra M$ be the map induced by the inclusion
$G_\mathrm{ad}\subset M$.  Then any line bundle on $G$ can be trivialized so
that the image of $\pi^*:\EV{\grl}^*\subset\grG(M,\calM_\nu)\lra\mk[G]$ is the
image of $c_{\VV{\grl}}$ and the claim follows from previous
remarks.
\end{proof}

Together with Lemma~\ref{lem:general-normality}, this gives the following

\begin{prop} \label{prp:normalita}
The variety $X_\Sigma$ is normal if and only if, for every $\mu \in \grL^+$
such that $\mu \leq \grl$, there exist $n>0$
and $\grl_1,\dots,\grl_n\in \Sigma$
such that
\[\VV{\mu + (n-1)\grl} \subset \VV{\grl_1}\otimes \cdots \otimes \VV{\grl_n} .\]

\end{prop}

\subsection{Remarks on tensor products}\label{ssez:prodottitensore}
By Proposition~\ref{prp:normalita}, in order to establish the normality
(or the non-normality) of $X_\Sigma$, we need some results on tensor
product decomposition.

\begin{lem} \label{lem:riduzionelevi}
Let $\grl, \mu, \nu$ be dominant weights and let $\grD'\subset \grD$ be such that
$\Supp_\grD(\grl + \mu - \nu) \subset \grD'$; let $L \subset G$ be the
standard Levi subgroup associated to $\grD'$. If $\pi \in \grL^+$, denote
by $\VL{\pi}$ the simple $L$-module of highest weight $\pi$. Then
\[
\VV{\nu} \subset \VV{\grl} \otimes \VV{\mu} \iff \VL{\nu} \subset \VL{\grl} \otimes \VL{\mu}.
\]
\end{lem}

\begin{proof}
If $\mathfrak a$ is any Lie algebra, denote $\goU(\mathfrak a)$
the corresponding universal enveloping algebra.

Suppose that $\VL{\nu} \subset \VL{\grl} \otimes \VL{\mu}$; fix
maximal vectors $v_\grl \in \VL{\grl}$ and $v_\mu \in \VL{\mu}$ for
the Borel subgroup $B\cap L \subset L$ and fix $p \in \goU(\mathfrak
l\cap\mathfrak u^-) \otimes \goU(\mathfrak l\cap\mathfrak u^-)$ such
that $p\,(v_\grl \otimes v_\mu) \in \VL{\grl} \otimes \VL{\mu}$ is a
maximal vector of weight $\nu$. Since $\VL{\grl} \otimes \VL{\mu}
\subset \VV{\grl} \otimes \VV{\mu}$, we only need to prove that
$p\,(v_\grl \otimes v_\mu)$ is a maximal vector for $B$ too. If $\gra
\in \grD'$ then we have $e_\gra p\,(v_\grl \otimes v_\mu) = 0$ by
hypothesis. On the other hand, if $\gra \in \grD \senza \grD'$, notice that
$e_\gra$ commutes with $p$, since by its definition $p$ is supported
only on the $f_\gra$'s with $\gra \in \grD'$. Since $v_\grl \otimes
v_\mu$ is a maximal vector for $B$, then we get
\[
    e_\gra p\,(v_\grl \otimes v_\mu) = p\, e_\gra (v_\grl \otimes v_\mu) = 0;
\]
thus $p\,(v_\grl \otimes v_\mu)$ generates a simple $G$-module of highest weight $\nu$.

Assume conversely that $\VV{\nu} \subset \VV{\grl} \otimes \VV{\mu}$ and fix $p
\in \goU(\mathfrak u^-) \otimes \goU(\mathfrak u^-)$ such that
$p\,(v_\grl \otimes v_\mu) \in \VV{\grl} \otimes \VV{\mu}$ is a maximal
vector of weight $\nu$. Since $\Supp_\grD(\grl + \mu - \nu) \subset \grD'$,
we may assume that the only $f_\gra$'s appearing in $p$ are those with
$\gra \in \grD'$; therefore $p\,(v_\grl \otimes v_\mu) \in \VL{\grl}
\otimes \VL{\mu}$ and it generates a simple $L$-module of highest
weight $\nu$.
\end{proof}

\begin{lem} \label{lem:traslazione}
Fix $\grl, \mu, \nu \in \grL^+$ such that $\VV{\nu} \subset \VV{\grl}
\otimes \VV{\mu}$. Then, for any $\nu' \in \grL^+$, it also holds
\[
\VV{\nu + \nu'} \subset \VV{\grl + \nu'} \otimes \VV{\mu}.
\]
\end{lem}

\begin{proof}
Fix a maximal vector $v_{\nu'} \in \VV{\nu'}$ and consider the
$U$-equivariant map
\[
\begin{array}{cccc}
\phi: & \VV{\grl} \otimes \VV{\mu} & \lra & \VV{\grl + \nu'}\otimes \VV{\mu}\\
& w_1 \otimes w_2 & \longmapsto & m_{\grl,\nu'}(w_1,v_{\nu'}) \otimes w_2
\end{array}
\]
The claim follows since, if $v_\nu \in \VV{\grl}\otimes \VV{\mu}$ is
a $U$-invariant vector of weight $\nu$, then $\phi(v_\nu) \in \VV{\grl +
\nu'} \otimes \VV{\mu}$ is a $U$-invariant vector of weight $\nu +
\nu'$.
\end{proof}

We now describe some more explicit results. When we deal
with explicit irreducible root systems, unless otherwise stated,
we always use the numbering of simple
roots and fundamental weights of Bourbaki \cite{Bo}.

In order to describe the simple subsets $\grS\subset \grL^+$ which give rise
to a non-normal variety $X_\grS$, we will make use of following lemma.

\begin{lem}\label{lem:BGno}\hfill

\begin{enumerate}
\item Let $G$ be of type ${\sf B}_r$. Then, for any $n$,
  $\VV{(n-1)\gro_1} \not \subset \VV{\gro_1}^{\otimes n}$.
\item Let $G$ be of type ${\sf G}_2$. Then, for any $n$,
  $\VV{\gro_1+(n-1)\gro_2} \not \subset \VV{\gro_2}^{\otimes n}$.
\end{enumerate}
\end{lem}

\begin{proof}
We consider only the first case, the second is similar. Fix a highest
weight vector $v_1 \in \VV{\gro_1}$. If $\gra$ is any simple root and
if $1 \leq s \leq r$, notice that $f_\gra$ acts non-trivially on
$f_{\gra_{s-1}}\cdots f_{\gra_1}v_1$ if and only if $\gra =
\gra_s$. The $T$-eigenspace of weight 0 in $\VV{\gro_1}$
is spanned by $v_0=f_{\gra_r}\cdots f_{\gra_1}v_1$, and similarly the
$T$-eigenspace of weight $(n-1)\gro_1$ in $\VV{\gro_1}^{\otimes n}$ is
spanned by $v_1^{\otimes i-1} \otimes v_0 \otimes v_1^{\otimes n-i}$,
where $1 \leq i \leq n$. Since the vectors
\[
    e_{\gra_r} (v_1^{\otimes i-1} \otimes v_0 \otimes v_1^{\otimes
      n-i}) = v_1^{\otimes i-1} \otimes (e_{\gra_r} v_0) \otimes
    v_1^{\otimes n-i}
\]
are linearly independent, there exists no maximal vector of
weight $(n-1)\gro_1$ in $\VV{\gro_1}^{\otimes n}$.
\end{proof}

Dual results will be needed to describe the subsets $\grS$ which
give rise to a normal variety $X_\grS$,
but before we need to introduce some further notation.

If $\Phi$ is an irreducible root system and $\grD$ is a basis for $\Phi$ we will denote by $\eta$ the highest root if $\Phi$
is simply laced or the highest short root if $\Phi$ is not simply
laced. For the convenience of the reader we list the highest short
root of every irreducible root system in Table~\ref{tab:hsr}.

\begin{table}\caption{}\label{tab:hsr}
\begin{center}
\begin{tabular}{c|c}
type of $\Phi$ & highest short root \\ \hline $\sfA_r$ &
$\gra_1+\cdots+\gra_r=\omega_1+\omega_r$\\ $\sfB_r$ &
$\gra_1+\cdots+\gra_r=\omega_1$\\ $\sfC_r$ &
$\gra_1+2(\gra_2+\cdots+\gra_{r-1})+\gra_r=\omega_2$\\ $\sfD_r$ &
$\gra_1+2(\gra_2+\cdots+\gra_{r-2})+\gra_{r-1}+\gra_r=\omega_2$\\ $\sfE_6$
& $\gra_1+2\gra_2+2\gra_3+3\gra_4+2\gra_5+\gra_6=\omega_2$\\ $\sfE_7$ &
$2\gra_1+2\gra_2+3\gra_3+4\gra_4+3\gra_5+2\gra_6+\gra_7=\omega_1$\\ $\sfE_8$ &
$2\gra_1+3\gra_2+4\gra_3+6\gra_4+5\gra_5+4\gra_6+3\gra_7+2\gra_8=\omega_8$\\ $\sfF_4$
& $\gra_1+2\gra_2+3\gra_3+2\gra_4=\omega_4$\\ $\sfG_2$ &
$2\gra_1+\gra_2=\omega_1$\\
\hline
\end{tabular}
\end{center}
\end{table}

Recall the condition $(\star)$ defined in the introduction: a dominant weight
$\grl$ satisfies $(\star)$ if, for every non-simply laced connected component $\grD'\subset \grD$,
if $\Supp(\grl)\cap \grD'$ contains a long root then it
contains also the short root which is adjacent to a
long simple root.

\begin{dfn}\label{twin}
If $\grD' \subset \grD$ is a non-simply laced connected component,
order the simple roots in $\grD'= \{ \gra_1, \ldots,
\gra_r\}$ starting from the
extreme of the Dynkin diagram of $\grD'$ which contains a long root and denote $\gra_q$ the first short root in $\grD'$. If $\grl$ is a dominant weight such that $\gra_q\not \in \Supp(\grl)$ and such that $\Supp(\grl)\cap \grD'$ contains a long root, denote $\gra_p$ the last long root which occurs in $\Supp(\grl)\cap \grD'$; for instance, if $\grD'$ is not of type $\sfG_2$, then the numbering is as follows:
\[
\begin{picture}(9000,1800)(2000,-900)
           \put(0,0){\multiput(0,0)(3600,0){2}{\circle*{150}}\thicklines\multiput(0,0)(2500,0){2}{\line(1,0){1100}}\multiput(1300,0)(400,0){3}{\line(1,0){200}}}
           \put(3600,0){\multiput(0,0)(3600,0){2}{\circle*{150}}\thicklines\multiput(0,0)(2500,0){2}{\line(1,0){1100}}\multiput(1300,0)(400,0){3}{\line(1,0){200}}}
           \put(7200,0){\multiput(0,0)(1800,0){2}{\circle*{150}}\thicklines\multiput(0,-60)(0,150){2}{\line(1,0){1800}}\multiput(1050,0)(-25,25){10}{\circle*{50}}\multiput(1050,0)(-25,-25){10}{\circle*{50}}}
           \put(9000,0){\multiput(0,0)(3600,0){2}{\circle*{150}}\thicklines\multiput(0,0)(2500,0){2}{\line(1,0){1100}}\multiput(1300,0)(400,0){3}{\line(1,0){200}}}
           \put(-150,-700){\tiny $\alpha_1$}
           \put(3450,-700){\tiny $\alpha_p$}
           \put(8850,-700){\tiny $\alpha_q$}
           \put(12450,-700){\tiny $\alpha_r$}
\end{picture}
\]
The \textit{little brother} of $\grl$ with respect to $\grD'$ is the dominant weight
\[
\grl_{\grD'}^\yb = \grl - \sum_{i=p}^q \gra_i =
\left\{ \begin{array}{ll}
		\grl-\omega_1+\omega_2 & \textrm{ if $G$ is of type $\sf{G}_2$} \\
		\grl + \gro_{p-1} - \gro_{p}  + \gro_{q+1} & \textrm{ otherwise}
\end{array} \right.
\]
where $\gro_i$ is the fundamental weight associated to $\gra_i$ if $1\leq i \leq r$, while $\gro_0 = \gro_{r+1} = 0$.
The set of the little brothers of
$\grl$ will be denoted by $\YB(\grl)$; notice that $\YB(\grl)$ is empty
if and only if $\grl$ satisfies condition $(\star)$ of Theorem A.
For convenience, define
$\overline \YB(\grl)=\YB(\grl)\cup\{\grl\}$, while if $\grD$ is connected and non-simply laced set $\grl^\yb = \grl_\grD^\yb$.
\end{dfn}

\begin{lem} \label{lem:eta}
Assume $G$ to be simple and let $\grl\in \grL^+\senza\{0\}$.
Denote $\eta$ the highest root of $\Phi$ if the latter is simply laced or the highest short root otherwise.
\begin{enumerate}
\item If $\grl$ satisfies the condition $(\star)$ then
$$ \VV{\grl} \subset \VV{\eta} \otimes \VV{\grl}. $$
\item If $\grl$ does not satisfy the condition $(\star)$ and if
$\grl^\yb$ is the little brother of $\grl$ then $$ \VV{\grl} \subset \VV{\eta} \otimes \VV{\grl^\yb}. $$
\end{enumerate}
\end{lem}

\begin{proof}
If $\grD$ is simply laced, then
$\VV{\eta}\simeq\gog$ is the adjoint representation: in this case the claim follows
straightforward by considering the map $\gog \otimes \VV{\grl} \to \VV{\grl}$
induced by the $\gog$-module structure on $\VV{\grl}$, which is non-zero
since $\grl$ is non-zero.

Suppose now that $\grD$ is not simply laced.
If $\grl$ satisfies condition $(\star)$, then
by Lemma~\ref{lem:traslazione} it is enough to study the case
$\grl = \omega_\gra$ where $\gra$ is a short simple root:

\emph{Type $\sfB_r$}: $\VV{\omega_r}\subset\VV{\omega_1}\otimes \VV{\omega_r}$.

\emph{Type $\sfC_r$}: $\VV{\omega_i}\subset\VV{\omega_2}\otimes \VV{\omega_i}$, with $i<r$.

\emph{Type $\sfF_4$}: $\VV{\omega_3}\subset\VV{\omega_4}\otimes \VV{\omega_3}$ and $\VV{\omega_4}\subset\VV{\omega_4}\otimes \VV{\omega_4}$.

\emph{Type $\sfG_2$}: $\VV{\omega_1}\subset\VV{\omega_1}\otimes \VV{\omega_1}$.

If $\grl$ does not satisfy condition $(\star)$,
by Lemma~\ref{lem:traslazione} we can assume
that $\grl=\gro_\gra$ with $\gra$ a long root:

\emph{Type $\sfB_r$}: $\VV{\omega_i}\subset\VV{\omega_1}\otimes \VV{\omega_{i-1}}$, if $1<i<r$, and $\VV{\omega_1}\subset\VV{\omega_1}\otimes \VV{0}$.

\emph{Type $\sfC_r$}: $\VV{\omega_r}\subset\VV{\omega_2}\otimes \VV{\omega_{r-2}}$.

\emph{Type $\sfF_4$}: $\VV{\omega_1}\subset\VV{\omega_4}\otimes \VV{\omega_4}$ and $\VV{\omega_2}\subset\VV{\omega_4}\otimes \VV{\omega_1+\omega_4}$.

\emph{Type $\sfG_2$}: $\VV{\omega_2}\subset\VV{\omega_1}\otimes \VV{\omega_1}$.

The above mentioned inclusion relations for tensor products are essentially known: let us treat the case of type $\sfC_r$ with $\grl=\omega_i$ and $i<r$, the other cases are easier or can be checked directly.

Let $v_0$ be a highest weight vector of $\VV{\omega_2}$ and $w_0$ be a highest weight vector of $\VV{\omega_i}$.
Let $f$ be the following product (in the universal enveloping algebra $\mathfrak U(\mathfrak u^-)$)
\[f=f_{\gra_i}\cdots f_{\gra_1}\cdot f_{\gra_{i+1}}\cdots f_{\gra_{r-1}}\cdot f_{\gra_r}\cdots f_{\gra_2},\]
and consider all the factorizations $f = p\cdot q$ such that $p,q \in\mathfrak U(\mathfrak u^-)$. If $\grb_1,\ldots,\grb_j\in\grD$, set
$$
	\,^\mathrm r(f_{\grb_1}\cdots f_{\grb_j})=(-1)^j2^\delta f_{\grb_j}\cdots f_{\grb_1},
$$
where $\delta$ equals 0 (resp.\ 1) if $\alpha_i$ occurs an even (resp.\ odd) number of times in $\{\grb_1,\ldots,\grb_j\}$. Then it is easy to check that the vector
\[\sum_{p\cdot q=f} p.v_0 \otimes \,^\mathrm r\!q.w_0\]
is a $U$-invariant vector in $\VV{\omega_2}\otimes\VV{\omega_i}$ of $T$-weight $\omega_i$.
\end{proof}

If the Dynkin diagram of $G$ is not simply laced we will need some
further properties of tensor products.

If $\grD$ is connected but not simply laced,
we will denote by $\alpha_S$
the short simple root that is adjacent to a long simple root $\alpha_L$;
moreover, we will denote the associated
fundamental weights by $\omega_S$ and $\omega_L$.
Finally, define $\zeta$ as the sum of all simple
roots and notice that $\omega_S + \zeta$ is dominant.

\begin{lem}\label{lem:zeta}
Let $\grl$ be a non-zero dominant weight.
\begin{enumerate}
\item If $G$ is of type $\sfF_4$ or $\sfC_r$ ($r\geq 3$) and if $\Supp(\grl)$ contains a long root then
$$\VV{\grl+\omega_S} \subset \VV{\zeta + \omega_S} \otimes \VV{\grl}.$$
\item If $G$ is of type $\sfG_2$ and if $\grl$ does not satisfy $(\star)$
  then $$ \VV{\grl+\omega_1} \subset \VV{\omega_2} \otimes \VV{\grl^\yb}. $$
\item If $G$ is of type $\sfG_2$ and if $\gra_S \in \Supp(\grl)$
  then $$ \VV{\grl+\omega_1} \subset \VV{\omega_2} \otimes \VV{\grl}. $$
\end{enumerate}
\end{lem}

\begin{proof}
By Lemma~\ref{lem:traslazione} it is enough to check the statements for $\grl=\omega_\gra$
with $\gra$ a long root in the first two cases and $\gra=\gra_S$ in the last case.

\emph{Type $\sfC_r$}: by Lemma~\ref{lem:traslazione} it is enough to check that
$\VV{\omega_{r-1}}\subset \VV{\omega_1} \otimes \VV{\omega_r}$.

\emph{Type $\sfF_4$}: we have $\grl=\omega_1$ or $\grl=\omega_2$ and $\omega_S+\zeta=\omega_1+\omega_4$.

\emph{Type $\sfG_2$}: we have $\grl=\omega_2$ and $\grl^\yb=\omega_1$ in point (2) and
$\grl=\omega_1$ in point (3).
\end{proof}

\subsection{Normality and non-normality of $X_\Sigma$}\label{ssez:normalita}

We are now able to state the main theorem.


\begin{teo}\label{teo:normalita}Let $\Sigma$ be a simple set of dominant weights and let
$\grl$ be its maximal element. The variety $X_\Sigma$
 is normal if and only if $\Sigma \supset \YB(\grl)$.
\end{teo}

Theorem~A stated in the introduction follows immediately by
considering the case $\Sigma=\{\grl\}$. The remaining part of this
section will be devoted to the proof of Theorem~\ref{teo:normalita}.
The general
strategy will be based on Proposition~\ref{prp:normalita} and will
proceed by induction on the dominance order of weights. The ingredients
of this induction will be the results proved in
section~\ref{ssez:prodottitensore}
together with the description of the dominance order given by J.~Stembridge in \cite{St}:
the dominance order between dominant weights is generated by pairs which
differ by the highest short root for a subsystem of the root system.

If $K$ is a subset of $\grD$, denote $\Phi_K\subset \Phi$ the
associated root subsystem and, in case $K$ is connected, denote by $\eta_K$ the
corresponding highest short root. Moreover, if $\grb =
\sum_{\gra\in\grD}n_\gra \gra$, set $\grb|_K=\sum_{\gra\in
  K}n_\gra \gra$. The result of \cite{St} that we will use is the following.

\begin{lem}[{\cite[Lemma 2.5]{St}}] \label{lem:stembridge}
Let $\grl,\mu$ be two dominant weights with $\grl>\mu$; set $I =
\Supp_\grD(\grl-\mu)$. Let $\Phi_K$ be an
irreducible subsystem of $\Phi_I$ (where $K\subset I$).
\begin{enumerate}[\indent (a)]
    \item If $\langle (\grl-\mu)\ristretto_K, \gra^\vee \rangle \geq 0$
      for all $\gra \in K \cap \Supp(\mu)$, then $\mu+\eta_K \leq \grl$.
    \item If in addition $\langle \mu + \eta_K, \gra^\vee
      \rangle \geq 0$ for all $\gra \in I \senza K$, then $\mu +
      \eta_K \in \grL^+$.
\end{enumerate}
\end{lem}

The next two lemmas are the main steps of our induction.

\begin{lem}\label{lem:costruzioneK}
Suppose that $\Phi$ is irreducible; let $\grl,\mu\in\grL^+$ such that
$\grl>\mu$ and $\Supp_\grD(\grl-\mu)=\grD$. Assume that either $\Phi$ is simply laced, or there exists a short root $\gra \in \Supp(\grl)$ such that $\langle  \grl-\mu,\gra^\cech \rangle \geq 0$, or $\gra_S \notin \Supp(\mu)$.
Then there exists a connected subset $K$ of $\grD$ such that
\begin{enumerate}[\indent i)]
\item $\mu+\eta_K\leq \grl$;
\item $\mu+\eta_K\in \grL^+$;
\item $K\cap\Supp(\grl)\neq \vuoto$.
\end{enumerate}
\end{lem}

\begin{proof}
Set $K_1= \{\gra\in \grD\st \langle\grl-\mu,\gra^\cech\rangle\geq 0\}$.
Since $\grl>\mu$ we have that $K_1\cap\Supp(\grl)$ is non-empty.
Notice also that $\Supp(\mu)\supset \grD\senza K_1$. Define $K$ as follows:
\begin{itemize}
	\item[a)] If $\Phi$ is simply laced, let $K$ be a connected component of $K_1$
which intersects $\Supp(\grl)$.
	\item[b)] If $\gra\in\Supp(\grl)$ is a short root such that
$\langle \grl-\mu, \gra^\cech \rangle \geq 0$
let $K$ be the connected component of $K_1$
containing $\gra$.
	\item[c)] If $\Phi$ is not simply laced and there does not exist a short root
$\gra$ as in b), let $K$ be a connected component of $K_1$ which intersects
$\Supp(\grl)$.
\end{itemize}

Properties $i)$ and $iii)$ are then easily verified by Lemma~\ref{lem:stembridge}(a)
 and by construction.

To prove $ii)$ notice that, if $\Phi$ is not simply laced, by the construction
of $K$ it follows that if $\gra_L \in K$ then $\gra_S \in K$ as well:
indeed, $K$ is a connected component of $K_1$ and if there is no short root
$\gra$ as in b) then $\gra_S \not \in \Supp(\mu)$ implies $\gra_S \in K_1$.
By the description of highest short roots in Table~\ref{tab:hsr}
we deduce that, if $\gra \in K \senza K^{\circ}$,
then the respective coefficient in $\eta_K$ is $1$:
hence $\langle \eta_K, \gra^\vee \rangle = -1$ for all $\gra \in \partial K$ and,
since $\Supp(\mu)\supset \grD \senza K_1 \supset \partial{K}$, we get $\mu + \eta_K \in \grL^+$.
\end{proof}

In order to proceed with the induction, in the next lemma we will need
to consider the condition $(\star)$ also for a Levi subgroup of $G$.
If $K\subset \grD$ let $L_K$ be the associated standard
Levi subgroup; we say that $\grl\in\grL^+$ satisfies condition $(\star_K)$ if, for every
non-simply laced connected component $K'$ of $K$ such that $\Supp(\grl)\cap K'$ contains a
long root, $\Supp(\grl)\cap K'$ contains also the short root adjacent to a
long root. Notice that if $\grl$ satisfies $(\star)$ then it
also satisfies $(\star_K)$ for all $K\subset \grD$.

Similarly we can also define the little brother
of a dominant weight w.r.t.\ the Levi subgroup $L_K$:
if $K'$ is a connected component of $K$
such that $\grl$ does not satisfy $(\star_{K'})$, define
the little brother $\grl_{K'}^\yb$ w.r.t.\ $K'$
as in Definition \ref{twin} and denote by $\YB_K(\grl)$ the set
of little brothers of $\grl$ constructed in this way.
Notice that if $K'$ is a connected component of $K$
such that $\grl$ does not satisfy $(\star_{K'})$ and
if $\grD'$ is the connected component of $\grD$ containing $K'$, then
$\grl$ does not satisfy $(\star_{\grD'})$ as well and
$\grl_{K'}^\yb=\grl_{\grD'}^\yb $. In particular $\YB_K(\grl)\subset
\YB(\grl)$.

\begin{lem}\label{lem:induzione}
Assume $G$ to be simple and let $\grl,\mu$ be two dominant weights
such that $\grl>\mu$ and $\Supp_\grD(\grl-\mu)=\grD$. Then there
exist $\mu'\in\grL^+$ and $\grl'\in \overline \YB(\grl)$
such that $\mu<\mu'\leq\grl$ and $$ \VV{\mu+\grl}\subset
\VV{\mu'}\otimes \VV{\grl'}. $$
\end{lem}

\begin{proof}
Suppose first that either $\Phi$ is simply laced or $\gra_S \notin \Supp(\mu)$ or there
exists a short root $\gra$ in $\Supp(\grl)$ such that $\langle
 \grl-\mu, \gra^\cech \rangle \geq 0$. Take $K$ as in
Lemma~\ref{lem:costruzioneK} and set $\mu'=\mu+\eta_K$:
then by Lemma~\ref{lem:eta} together with Lemma~\ref{lem:traslazione}
we get
\[
    \VK{\mu + \grl} \subset \VK{\mu'} \otimes \VK{\grl'}
\]
with $\grl'\in\overline \YB_K(\grl)$. The claim follows by Lemma
\ref{lem:riduzionelevi} together with the inclusion $\YB_K(\grl)\subset \YB(\grl)$.

Suppose now that $\Phi$ is not simply laced, that $\gra_S \in \Supp(\mu)$ and that there
is no short root $\gra \in \Supp(\grl)$ such that $\langle
  \grl-\mu, \gra^\cech \rangle \geq 0$.
  Since $\grl>\mu$ there exists $\gra \in \Supp(\grl)$
   such that $\langle \grl-\mu , \gra^\vee\rangle \geq 0$:
therefore, $\Supp(\grl)$ contains at
least a long root. Set $\mu'=\mu + \zeta$; notice that
$\mu'\leq \grl$ and that $\mu'$ is dominant. The claim follows then by
Lemma~\ref{lem:eta} and by Lemma~\ref{lem:traslazione} if $\Phi$ is of type $\sfB$, while if $\Phi$ is of type $\sfC$, $\sfF_4$ or $\sfG_2$ it follows by
Lemma~\ref{lem:zeta} and by Lemma~\ref{lem:traslazione}.
\end{proof}

\begin{proof}[Proof of Theorem~\ref{teo:normalita}]
We prove first that the condition is necessary.  Assume that there
exists a little brother $\mu = \grl_{\grD'}^\yb$ of $\grl$ which is not in
$\Sigma$.  We prove that for every positive $n$ and for every choice
of weights $\grl_1,\dots,\grl_n \in \Sigma$ the module
$\VV{\mu+(n-1)\grl}$ is not contained in
$\VV{\grl_1}\otimes \cdots \otimes \VV{\grl_n}$.

We proceed by contradiction. Assume there exist weights $\grl_1,\dots,\grl_n$ as above and notice
that any of them satisfies $\mu\leq \grl_i\leq\grl$:
indeed, $\grl-\mu=n\grl-(\mu+(n-1)\grl)\geq n\grl-\sum\grl_{i}\geq \grl-\grl_{i}$ for every $i$.
Therefore $\Supp_\Delta (\sum \grl_i -(\mu+(n-1)\grl) ) \subset \Supp_\Delta (\grl - \mu)$.
By Definition \ref{twin} together with Lemma~\ref{lem:riduzionelevi}, it is enough to analyse
the case $G$ of type $\sfB_r$ and $\Supp(\grl) = \{\gra_1\}$ or $G$
of type $\sfG_2$ and $\Supp(\grl) = \{\gra_2\}$.
We analyse these two cases separately.

\emph{Type $\sfB_r$}: we have $\grl = a \omega_1$, $\mu=(a-1)\omega_1$
and $\mu+(n-1)\grl=(na-1)\omega_1$.  If $a=1$ we notice that there are
no dominant weights between $\grl$ and $\mu$. So the only possibility
is $\grl_i=\grl=\omega_1$ for all $i$ and this is in contradiction
with Lemma~\ref{lem:BGno}. If $a>1$, notice that there is only one dominant weight between
$\grl$ and $\mu$, namely $\nu=\grl-\gra_1=(a-2)\omega_1+\omega_2$;
hence for all $i$ it must be $\grl_i=\grl$ or $\grl_i=\nu$.
Since $\sum \grl_i \geq \mu+(n-1)\grl$, at most one $\grl_i$ can be equal to $\nu$;
therefore $\VV{\mu+(n-1)\grl}\subset \VV{\grl}^{\otimes n}$ or
$\VV{\mu+(n-1)\grl}\subset \VV{\nu} \otimes \VV{\grl}^{\otimes (n-1)}$.
In the first case we obtain
$$ \VV{(na-1)\omega_1}=\VV{\mu+(n-1)\grl}\subset \VV{\grl}^{\otimes
  n}\subset \VV{\omega_1}^{\otimes na},$$
  against Lemma~\ref{lem:BGno}. In the second case we notice that
$\VV{\omega_2} = \grL^2 \VV{\omega_1} \subset \VV{\omega_1}^{\otimes 2}$, hence
$\VV{\nu}\subset \VV{(a-2)\omega_1}\otimes \VV{\omega_2} \subset \VV{\omega_1}^{\otimes a}$
and we can conclude as in the first case.

\emph{Type $\sfG_2$}: we have $\grl = a \omega_2$, $\mu=\omega_1+(a-1)\omega_2$ and
we proceed as in the previous case.

\

We now prove that the condition is sufficient, showing that
for every dominant weight $\mu \leq \grl$ there exist $n>0$ and weights
$\grl_1,\dots,\grl_n \in \overline \YB(\grl)$ such that
$\VV{\mu+(n-1)\grl}\subset\VV{\grl_1}\otimes \cdots\otimes
\VV{\grl_n}$. To do this, we proceed by decreasing
induction with respect to the dominance order.

If $\mu = \grl$ then the claim is clear, so we assume $\mu < \grl$.
Let $\grl - \mu = \grb_1+\dots+\grb_m$ where $\Supp_\grD(\grb_i)$ are the connected
components of $\Supp_\grD(\grl-\mu)$.  Set $K= \Supp_\grD(\grb_1)$ and
$\grb'=\grb_2+\dots+\grb_m$.
Notice that $\mu+\grb_1$ is dominant: indeed if $\gra \not \in \ol{K}$ then
$\langle \mu + \beta_1 , \gra^\vee \rangle =  \langle \mu , \gra^\vee \rangle \geq 0$, while 
if $\gra \in \ol{K}$ then
$\langle \mu + \beta_1 , \gra^\vee \rangle = \langle \grl - \grb',
\gra^\vee \rangle \geq  \langle \grl , \gra^\vee \rangle \geq 0$.
Notice moreover that, if $\nu\in \overline \YB_K(\mu+\grb_1)$, then $\nu+\grb' \in \overline
\YB(\grl)$. By Lemma~\ref{lem:induzione} applied to the semisimple part of the
Levi $L=L_K$ associated to $K$, there exists a weight $\mu'$
which is dominant with respect to $K$ such that $\mu < \mu'\leq \mu+\grb_1$
and there exists $\nu \in \overline \YB_K(\mu+\grb_1)$ which satisfy
$$\VL{\mu+\grb_1+\mu}\subset \VL{\mu'} \otimes \VL{\nu}.$$
By tensoring with $\VL{\grb'}$, which is a one dimensional representation,
we get $\VL{\mu+\grl}\subset \VL{\mu'} \otimes \VL{\grl'}$ with
$\grl'=\nu+\grb'\in \overline \YB(\grl)$.
Since $\langle \mu',\gra^\vee \rangle \geq \langle \mu+\grb_1,\gra^\vee \rangle$
for every $\gra \not \in K$, $\mu'$ is a dominant weight;
by Lemma~\ref{lem:riduzionelevi} we get then $\VV{\mu+\grl}\subset \VV{\mu'} \otimes \VV{\grl'}$
and we may apply the induction on $\mu'\leq \grl$.
Therefore there exist weights
$\grl_1,\dots,\grl_n \in \overline \YB(\grl)$ such that
$\VV{\mu'+(n-1)\grl}\subset\VV{\grl_1}\otimes \cdots\otimes
\VV{\grl_n}$. Finally by Lemma~\ref{lem:traslazione} we conclude
$$\VV{\mu+n\grl}\subset \VV{\mu'+(n-1)\grl}\otimes \VV{\grl'} \subset
\VV{\grl_1}\otimes \cdots\otimes \VV{\grl_n} \otimes \VV{\grl'}.$$
\end{proof}

\section{Smoothness}

In this section we will study the variety $\wt{X}_\grl$; in particular
we will give necessary and sufficient conditions on $\Supp(\grl)$ for
its $\mathbb{Q}$-factoriality and for its smoothness.

Thanks to Lemma \ref{lem:immersioni}, we may assume that $G$ is a
simple group. Indeed suppose $\grD = \cup_{i=1}^n \grD_i$ is the
decomposition in connected components and write $\grl = \grl_1 +
\ldots + \grl_n$ with $\Supp(\grl_i)\subset \grD_i$: correspondingly
we get a decomposition $X_\grl = X_{\grl_1} \times \ldots \times
X_{\grl_n}$, and every $X_{\grl_i}$ is an embedding of the
corresponding simple factor of $G_\mathrm{ad}$ if $\grl_i \neq 0$ or a
point if $\grl_i = 0$. From now on, we will therefore assume that
$\Phi$ is an irreducible root system.

By the Bruhat decomposition, the group $G_\mathrm{ad}$ has an open
$B\times B^-$-orbit; therefore it is a spherical $G\times
G$-homogeneous space. Following the general theory of spherical
embeddings (see \cite{Kn}), its simple normal embeddings are
classified by combinatorial data called the \textit{colored
  cones}. Here we will skip an overview of such theory, and we will
simply recall the definition of the colored cone in the particular
case of a simple normal embedding of $G_\mathrm{ad}$.

Recall that a normal variety $X$ is said $\mQ$-\textit{factorial} if,
given any Weil divisor $D$ of $X$, there exists an integer $n\neq 0$
such that $nD$ is a Cartier divisor.
In subsection~\ref{sez smooth-notaz}, we will explicitly describe the
colored cone of $\wt{X}_\grl$; then in subsection~\ref{sez
  Q-factoriality} we will study $\mathbb{Q}$-factoriality of
$\wt{X}_\grl$ following \cite{Br2}. Finally, in
subsection~\ref{sezione smoothness}, we will use
Theorem~\ref{teo:normalita} together with the description of the
colored cone of $\wt{X}_\grl$ to make more explicit the criterion of
smoothness given in \cite{Ti} in the case of a linear projective
compactification of a reductive group.

\subsection{The colored cone of $\wt{X}_\grl$}\label{sez smooth-notaz}

Let $X$ be a simple normal compactification of $G_\mathrm{ad}$, call $Y$ its unique closed orbit. Set $\calD(G_\mathrm{ad})$ the set of $B\times B^{-}$-stable prime divisors of $G_\mathrm{ad}$ and $\calD(X)\subset \calD(G_\mathrm{ad})$ the set of divisors whose closure in $X$ contains $Y$.
Let $\calN(X)$ be the set of $G\times G$-stable prime divisors of $X$, so that the set of $B\times B^{-}$-stable prime divisors of $X$ is identified with $\calD(G_\mathrm{ad})\cup \calN(X)$.

Let $T_\mathrm{ad}\subset G_\mathrm{ad}$ be the image of $T$; then the character group $\calX(T_\mathrm{ad})$ coincides with the root lattice $\mathbb{Z}\grD$, while the cocharacter group $\calX^\vee(T_\mathrm{ad})$ coincides with the coweight lattice $\grL^\vee$. If $V$ is a simple $G\times G$-module denote by $V^{(B\times B^{-})}$ the subset of $B\times B^{-}$-eigenvectors.
Notice that $\mk(G_\mathrm{ad})^{(B\times B^{-})}/\mk^{*} \simeq \mathbb{Z}\grD$ and define a natural map $\rho : \calD(G_\mathrm{ad})\cup \calN(X) \to \grL^\vee$ by associating to a $B\times B^-$-stable prime divisor of $X$ the cocharacter associated to the rational discrete valuation induced by $D$. If $D\in \calN(X)$, then $\rho(D)$ is the opposite of a fundamental coweight, while if $D\in \calD(G_\mathrm{ad})$, then $\rho(D)$ is a simple coroot; moreover, $\rho$ is injective and $\rho(\calD(G_\mathrm{ad}))=\Delta^\vee$ (see \cite[\S~7]{Ti}).

Let $\calC(X)$ be the convex cone in $\grL^\vee_{\mathbb{Q}}$ generated by $\rho\big(\calD(X)\cup \calN(X)\big)$; by the general theory of spherical embeddings we have that $\calC(X)$ is generated by $\rho(\calD(X))$ together with the negative Weyl chamber of $\Phi$. The \textit{colored cone} of $X$ is then the couple $\big(\calC(X),\calD(X)\big)$: up to equivariant isomorphisms, it uniquely determines $X$ as a $G\times G$-compactification of $G_\mathrm{ad}$.

In the case of the compactification $\widetilde{X}_{\grl}$,
then $\rho(\calD(X)) = \grD^\vee \senza \Supp(\lambda)^\vee$
(see \cite[Theorem~7]{Ti}).

\subsection{$\mathbb{Q}$-factoriality}\label{sez Q-factoriality}

In order to give a necessary and sufficient condition for the $\mathbb{Q}$-factoriality of $\wt{X}_\grl$ we need to determine the set of extremal rays of the associated cone $\calC(\wt{X}_\grl)$.

\begin{lem}
If $\gra \in \grD\senza \Supp(\lambda)$, then $\gra^\vee$ generates an extremal ray of $\calC(\wt{X}_\grl)$.
\end{lem}

\begin{proof}
If a simple coroot $\gra^\vee \in \calC(\wt{X}_\grl)$ does not generate an extremal ray, then we can write
\[
	\gra^\vee = \sum_{\grb \in \grD \senza \{\gra\}} a_{\grb}\grb^{\vee} - \sum_{\grb \in \grD}
	b_{\grb}\omega_{\grb}^{\vee},
\]
with $a_\grb, b_\grb \geq 0$ for every $\grb$: this yields a contradiction since then it would be $\langle \gra,\gra^\vee \rangle \leq 0$.
\end{proof}

Recall that a convex cone is said to be \textit{simplicial} if it is generated by linearly independent vectors; the following proposition is a particular case of a characterization of $\mathbb{Q}$-factoriality that M.~Brion gave in \cite{Br2} in the general case of a spherical variety. We recall it in the case of our interest.

\begin{prop} [see {\cite[Proposition 4.2]{Br2}}] \label{lem Q fattor}
The variety $\wt{X}_\grl$ is $\mathbb{Q}$-factorial if and only if $\calC(\wt{X}_\grl)$ is simplicial.
\end{prop}

Therefore, since $\calC(\wt{X}_\grl)$ has maximal dimension, $\wt{X}_\grl$ is $\mathbb{Q}$-factorial if and only if the number of extremal rays of the associated cone equals the rank of $G$. To describe such rays we need to introduce some more notation; the description will be slightly more complicated if $\Phi$ is of type $\sfD$ or $\sfE$.

Denote $\grD^e$ the set of extremal roots of $\grD$ and set $\grD \senza \Supp(\grl) = \bigcup_{i=1}^n I_i$ the decomposition in connected components. Denote
\[
	I^e = \bigcup_{\substack{I_i \neq I_\mathsf{de} \\ I_i \cap \grD^e \neq \vuoto}} I_i,
\]
where $I_\mathsf{de}$ is defined as follows. If $\grD$ is of type $\sf{D}$ or $\sf{E}$, denote $\grg_\mathsf{de}$ the unique simple root which is adjacent to other three simple roots and, if it exists, denote $I_\mathsf{de} \subset \grD \senza \Supp(\grl)$ the unique connected component such that $\grg_\mathsf{de} \in I_\mathsf{de}$ and $|I_\mathsf{de} \cap \grD^e| = 1$, otherwise define $I_\mathsf{de}$ to be the empty set. Denote $I^\ast_\mathsf{de} \subset I_\mathsf{de}$ the minimal connected subset such that $\grg_\mathsf{de} \in I^\ast_\mathsf{de}$ and $I^\ast_\mathsf{de} \cap \grD^e \neq \vuoto$, or define it to be the empty set otherwise.
Finally define
\[
	J(\grl) = \big(\grD \senza (\ol{I^e} \cup I^\ast_\mathsf{de})\big) \cup \big(\grD^e \senza \Supp(\grl)\big).
\]

\begin{lem}\label{raggi estremali2}
The extremal rays of $\calC(\wt{X}_\grl)$ are
generated by the simple coroots $\alpha^{\vee}$ with $\alpha\in \grD\senza \Supp(\grl)$ and
by the opposite of fundamental coweights $-\omega_{\alpha}^{\vee}$
with $\alpha\in J(\grl)$.
\end{lem}

\begin{proof} Recall that $\calC(\wt{X}_\grl)$ is
generated by the simple coroots $\alpha^{\vee}$ with $\alpha\in \grD\senza \Supp(\grl)$ together with the
fundamental coweights $-\omega_{\alpha}^{\vee}$ with $\alpha\in \grD$ and that every coroot $\alpha^{\vee}$ with $\alpha\in \grD\senza \Supp(\grl)$
generates an extremal ray of $\calC(\wt{X}_\grl)$.

A coweight $-\omega_{\gra}^{\vee}$
does not generate an extremal ray if and only if it can be written as follows
\[
	-\omega_{\alpha}^{\vee}=\sum_{\grb \in K}
	a_{\grb}\grb^{\vee} - \sum_{\grb \in H}
	b_{\grb} \omega_{\grb}^{\vee}
\]
with $a_\grb>0 $ for every $\grb\in K$ and with $b_\grb >0$ for every $\grb\in H$,
for suitable non-empty subsets $K\subset \grD\senza \Supp(\grl)$ and $H \subset \grD \senza\{\gra\}$. Since the right member of the equality is negative against every simple root in $\partial{K}$,
we get $\partial{K} = \{\gra\}$. 

Notice that $K$ is connected. 
Indeed if $K' \subset K$ is a connected component then $\partial{K'} = \{\gra\}$
and $\sum_{\grb \in K'} a_{\grb} \langle \gra, \grb^{\vee} \rangle < 0$:
therefore if $K$ contains two connected components it must be 
$$
	\sum_{\grb \in K} a_{\grb} \langle \gra, \grb^{\vee} \rangle \leq -2.
$$
On the other hand $\langle \gra, \omega_{\grb}^{\vee} \rangle = 0$ for every $\grb \in H$,
therefore if $K$ is not connected it follows
\[
	-1 = - \langle \gra, \omega_{\alpha}^{\vee} \rangle 
	= \sum_{\grb \in K} a_{\grb} \langle \gra, \grb^{\vee} \rangle \leq -2.
\]

Since $\partial{K}$ is one single root, $K$ contains an extreme of $\grD$, thus we get
$K\subset I^{e} \cup I_\mathsf{de}$.
Suppose that $\grg_\mathsf{de} \in K \subset I_\mathsf{de}$:
then we get a contradiction since it would be $|\partial{K}|=2$.
Therefore we get $K\subset I^{e} \cup (I^\ast_\mathsf{de} \senza \{\grg_\mathsf{de}\})$ and $\gra \in \ol{I^{e}} \cup I^\ast_\mathsf{de}$. 
Such a subset $K$ cannot exist if $\gra \in \grD^{e}\senza \Supp(\grl)$,
otherwise it would be $K = \grD \senza\{\gra\}$ which intersects $\Supp(\grl)$.
We get then that every $-\gro_\gra^\vee$ with $\gra \in J(\grl)$ generates an extremal ray of $\calC(\wt{X}_\grl)$.

Suppose conversely that $\gra \not \in J(\grl)$. Then we can construct a connected subset $K\subset I^{e} \cup (I^\ast_\mathsf{de} \senza \{\grg_\mathsf{de}\})$ such that $\partial{K} = \{\gra\}$. If $\grg \in K\cap \grD^e$, consider the fundamental coweight $(\gro_\grg^K)^\vee$ associated to $\grg$ in the irreducible root subsystem associated to $K$: then we get
\[
	(\gro_\grg^K)^\vee = \sum_{\grb \in K}a_\grb \grb^\vee = \gro_\grg^\vee - m \gro^\vee_\gra,
\]
where $a_\grb>0$ are rational coefficients and where $m>0$ is an integer. Therefore $-\gro_\gra^\vee$ does not generate an extremal ray of $\calC(\wt{X}_\grl)$.
\end{proof}

\begin{prop}\label{Q-fattorialita}
The variety $\wt{X}_\grl$ is $\mathbb{Q}$-factorial if and only if the following conditions are fulfilled:
\begin{itemize}
\item[{\em i)}] $\Supp(\grl)$ is connected;
\item[{\em ii)}] If $\Supp(\grl)$ contains a unique element, then this element is an
  extreme of $\grD$;
\item[{\em iii)}] If $\grD$ is of type $\sfD$ or $\sfE$, then $\Supp(\grl)$ contains $\grg_\mathsf{de}$ and at least two simple roots adjacent to $\grg_\mathsf{de}$.
\end{itemize}
\end{prop}

\begin{proof}
By Proposition \ref{lem Q fattor} together with Lemma \ref{raggi estremali2} $\wt{X}_\grl$ is $\mathbb{Q}$-factorial if and only if $|\Supp(\grl)| = |J(\grl)|$.

Suppose that $\wt{X}_\grl$ is $\mathbb{Q}$-factorial.
Consider the dominant weight
$\grl' = \sum_{\gra \not \in I^e \cup I^\ast_\mathsf{de}} \gro_\gra$:
then $J(\grl') = J(\grl)$ and
\[
 |\grD| = |J(\grl)| + |\grD \senza \Supp(\grl)| \geq |J(\grl')| + |\grD \senza \Supp(\grl')| \geq |\grD|,
\]
which implies $\Supp(\grl) = \Supp(\grl')$. This shows $\grD \senza \Supp(\grl) = I^e \cup I^\ast_\mathsf{de}$, and we get the following decomposition of $J(\grl)$:
\[
	J(\grl) \cap \Supp(\grl) = \grD \senza (\ol{I^e}\cup I^\ast_\mathsf{de}), \qquad	\quad
	J(\grl) \senza \Supp(\grl) = \grD^e \senza \Supp(\grl).
\]

If $I_\mathsf{de} \neq \vuoto$, set $I_\mathsf{de} \cap \grD^e = \{\gra_\mathsf{de}\}$.
Define a surjective map $F: J(\grl) \senza \{\gra_\mathsf{de}\} \lra \Supp(\grl)$ as follows:
$F$ is the identity on $J(\grl) \cap \Supp(\grl)$,
while if $\gra \in J(\grl) \senza \Supp(\grl)$ consider the
connected component $K\subset \grD\senza \Supp(\grl)$ containing $\gra$ and
define $F(\gra)$ by the relation $\partial{K} = \{F(\gra)\}$:
since $\gra \neq \gra_\mathsf{de}$, it must be $|\partial{K}|=1$.
Therefore $F$ is well defined and it is surjective
since $\Supp(\grl) \senza J(\grl) = \partial{I^e}$.
Therefore $\grD \senza \Supp(\grl) = I^e$ and we get i).
Being surjective, $F$ has to be injective as well; this easily implies both ii) and iii).

Suppose conversely that $\Supp(\grl)$ is connected,
or equivalently that $\grD\senza \Supp(\grl) = I^e$:
then ii) and iii) imply $|\grD^e \senza \Supp(\grl)| = |\partial{I^e}|$.
This shows that
$\wt{X}_\grl$ is $\mathbb{Q}$-factorial, since then
$|J(\grl)| + |\grD\senza \Supp(\grl)| = |\grD|$.
\end{proof}

\begin{cor} \label{cor:raggi estremali}
If $\wt{X}_\grl$ is $\mathbb{Q}$-factorial, the extremal rays of $\calC(\wt{X}_\grl)$ are generated by:
\begin{itemize}
	\item[{\em i)}] the coroots $\alpha^{\vee}$ with $\alpha\in \grD\senza \Supp(\grl)$,
	\item[{\em ii)}] the coweights $-\omega_{\alpha}^{\vee}$ with $\alpha\in
				\Supp(\grl)^\circ \cup \big(\grD^e \senza \Supp(\grl)\big)$.
\end{itemize}
\end{cor}

\subsection{Smoothness}\label{sezione smoothness}

Suppose that $\grS=\{\grl,\grl_1,\ldots,\grl_s\}$ is a simple set
of dominant weights, where $\grl$ is the maximal one.
In this section we will prove the following generalization of Theorem B.

\begin{teo} \label{smooth Xsigma} The variety $X_{\Sigma}$ is smooth if and only if $X_{\lambda}$ is  normal, $\mathbb{Q}$-factorial and every connected component of  $\grD\senza\Supp(\lambda)$ has type $\sfA$.
\end{teo}

\begin{cor}
$X_{\Sigma}$ is smooth if and only if $X_{\lambda}$ is smooth.
\end{cor}

To prove Theorem~\ref{smooth Xsigma}, we will make use of a characterization of smoothness for arbitrary group compactifications given by D.~Timashev in \cite{Ti}. For convenience, we will use a generalization of it which can be found in \cite{Ru} in the more general context of symmetric spaces. We recall it in the case of a simple group compactification.

\begin{teo}[see {\cite[Theorem~2.2]{Ru}, \cite[Theorem 9]{Ti}}]\label{smooth-timashev}
The variety $\wt{X}_{\grl}$ is smooth if and only if the following conditions are fulfilled:
\begin{itemize}
\item[{\em i)}] All connected components of $\grD \senza \Supp(\lambda)$ are of type $\sfA$ and there
  are no more than $|\Supp(\lambda)|$ of them.
\item[{\em ii)}]  The cone $\calC(\wt{X}_\grl)$ is simplicial and it is generated by a
  basis of the coweight lattice $\Lambda^{\vee}$.
\item[{\em iii)}]  One can enumerate the simple roots in order of their positions
  at Dynkin diagrams of connected components
  $I_k = \{\alpha_{1}^{k},\ldots,\alpha_{n_{k}}^{k}\}$ of $\grD \senza \Supp(\grl)$ ,
  $k=1,\ldots,n$, and partition the basis of the free semigroup
  $\calC(\wt{X}_\grl)^{\vee}\cap \mathbb{Z}\grD$ into subsets
  $\{\pi_{1}^{k},\ldots,\pi_{n_{k}+1}^{k}\}$, $k=1,\ldots,p$, $p\geq n$,
  in such a way that
  $\langle\pi_{j}^{k}, (\alpha_{i}^{h})^\vee \rangle = \delta_{i,j}\delta_{h,k}$
  and $\pi_{j}^{k} - \frac{j}{n_{k}+1} \pi_{n_{k}+1}^{k}$
  is the $j$-th fundamental weight of the root system generated by
  $\{\alpha_{1}^{k},...,\alpha_{n_{k}}^{k}\}$ for all $j,k$.
\end{itemize}
\end{teo}

\begin{proof}[Proof of Theorem~\ref{smooth Xsigma}.]
First, we prove that the conditions are necessary;
since we only have to prove that $X_\grl$ is normal,
we may assume that $\grD$ is non-simply laced.
By Theorem~\ref{smooth-timashev} i), $\Supp(\grl)$ contains
at least one of the two simple roots $\alpha_S$, $\gra_L$;
suppose that $\Supp(\grl)$ contains $\alpha_L$ but not $\gra_S$.
Denote $K =\{\gra_1,\ldots,\gra_l\} \subset \grD\senza \Supp(\lambda)$ the connected component which contains $\gra_S$ and number its simple roots starting from $\gra_S$: therefore $\gra_1 = \gra_S$ and $\gra_l \in \grD^e$, moreover $\ol{K}$ is either of type ${\sf C}_{l+1}$ or of type $\sf{G}_2$.
Set $\gro^\vee = (l+1)(\gro^K_l)^\vee$, where $(\gro^K_l)^\vee$ is the fundamental coweight associated
to $\gra_l$ in the root subsystem $\Phi_K$ associated to $K$; then
\[
	\gro^\vee = \sum_{i=1}^{l} i \gra^\vee_{i} =
(l+1) \gro^\vee_{\gra_l} -m \gro_{\gra_L}^\vee.
\]
where $m = 2$ if $\ol{K}$ is of type ${\sf C}_{l+1}$ (with $l\geq 1$) and $m=3$ if $\ol{K}$ is of type ${\sf G}_2$.

If $\ol{K}$ is not of type ${\sf B}_2$, then $\grD$ is either of type ${\sf C}_r$ (with $r>2$) or of type ${\sf F}_4$ or of type ${\sf G}_2$ and every simple coroot $\grb^\vee \in \grD^\vee$ is a primitive element in $\Lambda^{\vee}$ (i.e. there does not exist $\pi^\vee\in \grL^\vee$ which satisfies $t\pi^\vee = \grb^\vee$ with $t>1$):
therefore by Lemma \ref{raggi estremali2} together with Theorem \ref{smooth-timashev} ii) $\{\gra^\vee_1, \ldots, \gra^\vee_l, \gro^\vee_{\gra_l} \}$ is part of a basis of $\grL^\vee$ and we get a contradiction since then the equality above would imply $\gro_{\gra_L}^\vee\not \in \grL^\vee$.
Otherwise $\ol{K}$ is of type ${\sf B}_2$, thus $\grD$ is of type ${\sf B}_r$ and $\frac{1}{2}\gra_S^\vee \in \grL^\vee$: then we get a contradiction since by Theorem \ref{smooth-timashev} iii) there exists $\pi \in \calC(\wt{X}_\grl)^{\vee}\cap \mathbb{Z}\grD$ such that $\langle \pi, \gra_S^\vee \rangle = 1$.

Let's prove now that conditions of Theorem~\ref{smooth-timashev}
are verified if $X_{\lambda}$ is normal, $\mathbb{Q}$-factorial
and $\grD\senza \Supp(\lambda)$ has type $\sfA$.
Set $N = \calC(\wt{X}_\grl) \cap \grL^\vee$ the monoid generated by the primitive elements of the extremal rays of $\calC(\wt{X}_\grl)$.

To prove condition i), it is enough to notice as in Proposition \ref{Q-fattorialita} that,
since $\Supp(\grl)$ is connected, we have $\grD \senza \Supp(\grl) = I^e$ and the number of its connected components equals $|\grD^e \senza \Supp(\grl)| \leq |J(\grl)| = |\Supp(\grl)|$.

To prove condition ii), let's show that,
if $\grb \in \grD \senza J(\grl) = \ol{I^e}\senza \grD^e$,
then $-\omega_{\grb}^{\vee} \in N$. Denote $I =\{\gra_1,\ldots,\gra_l\} \subset \grD\senza \Supp(\lambda)$ the connected component which contains $\grb$ in its closure and number its simple roots starting from the extreme of $I$ which is not an extreme of $\grD$; therefore $\gra_l \in \grD^e$.
Let $j$ be such that $\grb = \gra_j$ or set $j = 0$ if $\grb \in \Supp(\grl)$.
Set $K=\{\gra_{j+1},\ldots,\gra_{l}\}$ and set $\gro^\vee = (l-j+1)(\gro^K_l)^\vee$, where $(\gro^K_l)^\vee$ is the fundamental weight associated
to $\gra_l$ in the root subsystem $\Phi_K$ associated to $K$; then
\[
	\gro^\vee = \sum_{i=1}^{l-j} i \gra^\vee_{j+i} =
(l-j+1)\gro^\vee_{\gra_l} + \langle \grb,\gra_{j+1}^\vee \rangle \gro_{\grb}^\vee.
\]
Since $X_\grl$ is normal, by Theorem A we get $\langle \grb,\gra_{j+1}^\vee \rangle = -1$;
therefore by Corollary \ref{cor:raggi estremali} $-\gro_{\grb}^\vee \in N$.

Finally let's show that condition iii) holds.
Suppose that $K = \{\gra_1, \ldots, \gra_l \} \subset \grD\senza \Supp(\grl)$ is a connected component, where the simple roots in $K$ are numbered starting from the extreme of $K$ which is not an extreme of $\grD$, and define
\[
\pi_i^K =
\left\{ \begin{array}{ll}
		(\gra_i^\vee)^\ast & \textrm{ if $i\leq l$} \\
		(-\gro^\vee_{\gra_l})^\ast & \textrm{ if $i=l+1$}
\end{array} \right.
\]
where, if $\{v_1, \ldots, v_r\}$ is a basis of $\grL^\vee$,
$\{v_1^\ast, \ldots, v_r^\ast\}$ denotes the dual basis of $\grL$.
Therefore, if $\omega_{j}^{K}$ is the $j$-th fundamental weight of
$\Phi_{K}$, we have
$\omega_{j}^{K} = \pi^K_j - \frac{j}{l+1} \pi^K_{l+1}$.
\end{proof}

\section{Remarks and generalizations}

In this section we will consider the more general situation of compactifications of
symmetric varieties.

Let $G$ be as before and $\grs:G\to G$ an
involution of $G$. We denote by $H^\circ$ the subgroup of points fixed by
$\grs$ and by $H$ its normalizer. The notation is not completely coherent
with those of previous sections:
$G$ plays now the role that $G\times G$ played before, while $H^\circ$ has now the role that the diagonal of $G\times G$ had before.

Let $\Omega^+$ be the set of dominant weights $\grl$ such that
$\VV{\grl}$ has a non-zero vector fixed by $H^\circ$ and $\Omega$ the
sublattice of $\grL$ generated by $\Omega^+$. The monoid $\Omega^+$
(resp.\ the lattice $\Omega$) is in a natural way the set of
dominant weights (resp.\ the set of weights) of a (possibly non-reduced)
root system $\tPhi$, which is called the {\em restricted root system}.
For $\grl \in \Omega^+$ we can consider
the (unique) point $x_\grl \in \mP(\VV{\grl})$ fixed by $H$ and define $X_\grl$
as the closure of the $G$-orbit of $x_\grl$ in $\mP(\VV{\grl})$.

Proposition~\ref{prp:supporto} generalizes to this more general situation
without any further comment.

\subsection{Normality of $X_\grl$ and the closure of a maximal torus orbit.}
Let $T\subset G$ be a maximal torus such that the dimension of $TH$ is maximal
and let $Z_\grl = \ol{T\,x_\grl} \subset X_\grl$.
In \cite{Ru}, it is proved that when $X_\grl$ is normal then $Z_\grl$ also is normal.
The converse of this result does not hold in general.
Indeed $Z_\grl$ is always normal in the case of the $G\times G$-compactification of $G_\mathrm{ad}$.

\subsection{Generalization to symmetric varieties: normality}\label{ssez:simmetrichenormalita}
The wonderful compactification has been defined in the more general
situation of symmetric varieties
and the description of the normalization of $X_\grl$
generalizes thanks to the results contained in
\cite{CM} and \cite{CDM} (which generalize \cite{Ka} and \cite{DC}).
In particular, Lemma~\ref{lem:general-normality} holds here in general.
However, in the case of symmetric varieties we do not have a
clear description of the multiplication of sections as in
Lemma~\ref{lem:coefficientimatriciali}. In particular, we have no
analogue of Proposition~\ref{prp:normalita}.

One may wonder whether the normality of $X_\grl$ is equivalent to the
analogous combinatorial condition on the weight $\grl$, that is, $\grl$ satisfies condition $(\star)$ w.r.t.\ the root
system $\tPhi$; here is a counterexample.

Let $G$ be of type ${\sf B}_2$ and let $\grs$ be the involution of type B~I:
thus $G/H \simeq \mathrm{SO}(5)/\mathrm{S}\big(\mathrm{O}(3)\times \mathrm{O}(2)\big)$ and $\tDelta=2\grD$. 
Consider $\grl=2\gro_1\in\grO^+$; then $X_\grl$ is a normal
embedding of $G/H$.

Denote by $\leq_\grs$ the dominance order
w.r.t.\ the root system $\tPhi$ and suppose that $X_\grl$ is normal.
Then $\grl$ satisfies
\begin{quote}
\emph{for all $\mu\in\grO^+$ such that $\mu\leq_\grs\grl$ there
  exists $n\in\mN$ such that $\VV{\mu+(n-1)\grl}\subset
  \mathrm S^n(\VV{\grl})$.}
\end{quote}
If one assumes that the multiplication map is as generic as possible, then
also the converse is true.

\subsection{Generalization to symmetric varieties: smoothness}

In the setting of normal compactifications of symmetric varieties $G/H^\circ$, fix a maximal torus $T$ such that $TH^\circ$ has maximal dimension and a Borel subgroup $B\supset T$ such that $BH^\circ \subset G$ is dense. If $X$ is a simple normal compactification of $G/H$, denote $\calD(X)$ the set of $B$-stable and not $G$-stable prime divisors of $X$ which contain the closed orbit. Denote $\rho : \calD(X) \to \grO^\vee$ the map defined by the evaluation of functions;
by \cite[Proposition 1]{Vu} $\rho(\calD(X))$ is a basis of the restricted coroot system $\wt{\Phi}^\vee$.
Since the map $\rho$ is not always injective, following the criterion of $\mathbb{Q}$-factoriality in \cite{Br2} in order to generalize Proposition~\ref{Q-fattorialita} we only need to assume that $\rho$ is injective on $\calD(X)$, and the proof is the same.
Such proposition is true also for compactifications of $G/H^\circ$, and not only of $G/H$, since $\mathbb Q$-factoriality concerns no integrality questions.

Theorem~\ref{smooth Xsigma} also can be generalized to this setting with the same proof, but we do not have anymore the equivalence between property $(\star)$ and the normality of $X_\grl$. Thus the theorem has to be reformulated as follows (recall that a simple normal spherical variety is always quasi-projective).

\begin{teo}\label{smooth general}
A  simple normal compactification $X$ of $G/H$  is smooth if and only if it is $\mathbb{Q}$-factorial,  $\grD\senza\rho(\calD(X))$ satisfies $(\star)$ and every connected component of $\rho(\calD(X))$ has type $\sfA$.
\end{teo}


\begin{thebibliography}{KKLV}

\bibitem[Bo]{Bo} N.~Bourbaki, {\em \'El\'ements de
  math\'ematique. Fasc.\ XXXIV. Groupes et alg\`ebres de Lie.  Chapitres
  IV, V, VI}, Actualit\'es Scientifiques et Industrielles
  \textbf{1337}, Hermann Paris 1968.


\bibitem[Br]{Br2} M.~Brion, {\em Vari\'et\'es sph\'eriques et
  th\'eorie de Mori},
  Duke Math.\ J. \textbf{72} (1993) no.\ 2, 369--404.

\bibitem[CDM]{CDM} R.~Chiriv\`\i, C.~De Concini and A.~Maffei,
  {\em On
  normality of cones over symmetric varieties}, Tohoku Math.\ J. (2)
  \textbf{58}
  (2006) no.\ 4, 599--616.

\bibitem[CM]{CM} R.~Chiriv{\`{\i}} and A.~Maffei,
  {\em
  Projective normality of complete symmetric varieties},
  Duke Math.\ J. \textbf{122} (2004), 93--123.

\bibitem[D]{DC} C.~De Concini,
{\em Normality and non normality of certain semigroups and orbit
closures}, Algebraic transformation groups and algebraic varieties,
Encyclopaedia Math.\ Sci. \textbf{132}, Springer Berlin 2004, 15--35.

\bibitem[DP]{CP} C.~De Concini and C.~Procesi, {\em Complete
  symmetric varieties}, Invariant Theory, Lecture Notes
  in Math. \textbf{996}, Springer, Berlin, 1983, 1--44.

\bibitem[Ka]{Ka} S.S.~Kannan, {\em Projective normality of
  the wonderful compactification of semisimple adjoint groups},
  Math.\ Z. \textbf{239} (2002) 673--682.

\bibitem[Kn]{Kn} F.~Knop,
{\em The Luna-Vust theory of spherical embeddings}, Proceedings of the Hyderabad Conference on Algebraic Groups (Hyderabad, 1989), 225--249, Manoj Prakashan, Madras, 1991.

\bibitem[KKLV]{KKLV} F.~Knop, H.~Kraft, D.~Luna and T.~Vust,
{\em Local properties of algebraic group actions},
Algebraische Transformationsgruppen und Invariantentheorie, DMV Sem. \textbf{13},
Birkh\"auser, Basel, 1989, 63--75.


\bibitem[Ru]{Ru} A.~Ruzzi, {\em Smooth projective symmetric
  varieties with Picard number equal to one}.
  To appear in Internat.\ J.\ Math.

\bibitem[St]{St}J.R.~Stembridge, {\em The partial order of
  dominant weights}, Adv.\ Math. \textbf{136} (1998) no.\ 2,
  340--364.

\bibitem[Ti]{Ti}D.A.~Timashev, {\em Equivariant
  compactifications of reductive groups}, Sb.\ Math.
  \textbf{194} (2003) no.\ 3-4, 589--616.

\bibitem[Vu]{Vu} Th.~Vust,
{\em Plongements d'espaces sym\'etriques alg\'ebriques: une classification},
Ann. Scuola Norm. Sup. Pisa Cl. Sci. (4) \textbf{17} (1990), no.\ 2, 165--195.

\end{thebibliography}
\end{document}